\documentclass[10pt]{amsart}

\usepackage{amsmath,amssymb,amsfonts,amsthm}
\newtheorem{theorem}{Theorem}[section]
\newtheorem{maintheorem}{Theorem}

\newtheorem{secondtheorem}{Theorem}

\newtheorem{proposition}[theorem]{Proposition}
\newtheorem{corollary}[theorem]{Corollary}
\newtheorem{lemma}[theorem]{Lemma}
\theoremstyle{definition}
\newtheorem{definition}[theorem]{Definition}
\newtheorem{example}[theorem]{Example}
\newtheorem{remark}[theorem]{Remark}
\numberwithin{equation}{section}

\usepackage[english]{babel}
\usepackage[latin1]{inputenc}

\def\im{\imath}

\def\C{{\mathbb{C}}}

\def\P{{\mathbb{P}}}

\def\Z{{\mathbb{Z}}}

\def\FF{{\mathcal{F}}}
\def\BB{\textrm{BB}}
\def\tr{{\textrm{Tr}}}
\def\res{{\textrm{Res}}}
\def\cod{\textrm{codim}}
\def\a{{\alpha}}
\def\b{{\beta}}
\def\g{{\gamma}}
\def\ab{{\alpha\beta}}
\def\la{{\lambda}}
\def\si{{\sigma}}
\def\va{{\varphi}}

\def\o{\omega}

\def\Om{{\Omega}}

\def\O{{\mathcal{O}}}

\begin{document}
\title [On non-Kupka points of codimension one foliations]
{ On non-Kupka points of codimension one foliations on $\P^3$ }

\author[Calvo-Andrare, C\^orrea, Fern\'andez--P\'erez]
{Omegar Calvo-Andrade 
	\and Maur\'icio Corr\^ea
	\and Arturo Fern\'andez--P\'erez }

\thanks{Calvo-Andrade: FAPESP $n^o$ 2014/23594--6, CONACYT 262121}
\thanks{Corr\^ea: CAPES--DGU 247/11, CNPq 300352/2012--3 and PPM--00169--13}
\thanks{Fern\'andez--P\'erez: CAPES--Bolsista da CAPES--Brasilia/Brasil, IMPA - Brasil}

\dedicatory{To Jos\'e Seade in his 60 birthday}
\address{\emph{Calvo-Andrade.} CIMAT: Ap. Postal 402, \\ Guanajuato, 36000,\\ Gto. M\'exico}
\email{omegar.mat@gmail.com} 

\address{\emph{Corr\^ea Jr.}: Depto. de Mat.--ICEX
	Universidade Federal de Minas Gerais, UFMG}
\curraddr{Av. Ant\^onio Carlos 6627, 31270-901,
	Belo Horizonte-MG, Brasil.}
\email{mauriciomatufmg@gmail.com}

\address{\emph{Fern\'andez-P\'erez}: Depto. de Mat.--ICEX
	\\ Universidade Federal de Minas Gerais, UFMG}
\curraddr{Av. Ant\^onio Carlos 6627, 31270-901,
	Belo Horizonte-MG, Brasil.}
\email{arturofp@mat.ufmg.br}
\date{}

\maketitle

{\small\bf Abstract.} {\small\it We study the singular set of a codimension one holomorphic foliation on $\P^3$. We find a local normal form for these foliations near a codimension two component of the singular set that is not of Kupka type. We also determine the number of non-Kupka points immersed in a codimension two component of the singular set of a codimension one foliation on $\P^3$.
} 
\vskip .2in
{\small\bf Resumo.} {\small\it 
	Estudamos o conjunto singular de uma folhea\c{c}\~ao holomorfa de codimens\~ao um em  $\P^3$.
	Encontramos uma forma normal local para tais folhea\c{c}\~oes em torno de uma componente de codimens\~ao dois do seu conjunto singular que n\~ao \'e do tipo Kupka. Tamb\'em, determinamos o n\'umero de pontos n\~ao--Kupka imersos numa componente de codimens\~ao dois de uma folhea\c{c}\~ao de codimens\~ao um em  $\P^3$.}

\section{Introduction}

A \textit{regular codimension one holomorphic foliation} on a complex manifold $M$, can be defined by a triple 
$\{(\mathfrak{U},f_{\a},\psi_{\ab})\}$ where 
\begin{enumerate}
	\item[(i)] $\mathfrak{U}=\{U_{\a}\}$ is an open cover of $M$.
	\item[(ii)] $f_{\a}:U_{\a}\rightarrow\C$ is a holomorphic submersion for each $\alpha$.
	\item [(iii)] A family of biholomorphisms  
	$\{\psi_{\ab}:f_{\b}(U_{\ab})\to f_{\a}(U_{\ab})\}$ such that  
	\[
	\psi_{\ab}=\psi_{\b\a}^{-1},\quad f_{\b}|_{U_\a\cap U_\b}=\psi_{\b\a}\circ f_{\a}|_{U_\a\cap U_\b}
	\quad\mbox{and}\quad
	\psi_{\a\g}=\psi_{\ab}\circ\psi_{\b\g}.
	\]
\end{enumerate}
Since  
$df_{\a}(x)=\psi'_{\ab}(f_{\b}(x))\cdot df_{\b}(x),$ the set $ F=\displaystyle\bigcup_{\a} Ker(df_{\a})\subset TM$ is a subbundle. Also
$[\psi'_{\ab}(f_{\b})]\in \check{H}^{1}(\mathfrak{U},\O^{\ast})$
define a line bundle $N=TM/F$. The family of 1--forms $\{df_{\a}\},$ glue to a section  
$\o\in H^0(M,\Om^1(N))$. We get
\[
0\rightarrow F\rightarrow TM\xrightarrow{\{df_{\a}\}}N\rightarrow0, \quad
0\rightarrow\FF\rightarrow \Theta\xrightarrow{\{df_{\a}\}}\mathcal{N}\rightarrow0,\quad [\FF,\FF]\subset\FF
\]
where $\FF=\O(F),\, \Theta=\O(TM),\mbox{ and } \mathcal{N}=\O(N)$. We also have   
\[
\wedge^{n}TM^{\ast}=det(F^{\ast})\otimes N^{\ast},\qquad \Om^n_M:=K_M=
det(\FF^{\ast})\otimes \mathcal{N}^{\ast},\quad n=dim(M).
\]

\begin{definition} Let $M$ be a compact complex manifold with $\dim(M)=n$.
	A \textit{singular codimension one holomorphic foliation on $M$}, may be defined by one of the following ways:
	\begin{enumerate}
		\item A pair $\mathcal{F}=(S,\FF)$, where $S\subset M$ is an analytic subset of $\cod(S)\geq2$, and $\FF$ is a regular codimension one holomorphic foliation on $M\setminus S$.
		
		\item A class of sections $[\o]\in \P H^0(M, \Om^1(L))$ where $L\in Pic(M)$, such that 
		\begin{enumerate}
			\item[(i)] the singular set $S_{\o}=\{p\in M| \o_p=0  \}$ has $\cod(S_{\o})\geq2$.
			
			\item[(ii)] $\o \wedge d\o=0$ in $H^0(M,\Om^3(L^{\otimes 2}))$.
		\end{enumerate}
		We denote by $\mathcal{F}_{\o}=(S_\o,\FF_\o)$ the foliation represented by $\o$.
		\item An exact sequence of sheaves
		\[
		0\rightarrow \FF\rightarrow \Theta\rightarrow \mathcal{N}\rightarrow 0,\quad [\FF,\FF]\subset \FF
		\] 
		where $\FF$ is a reflexive sheaf of rank $rk(\FF)=n-1$ 
		with torsion free quotient $\mathcal{N} \simeq \mathcal{J}_S\otimes L,$ where $\mathcal{J}_S$ is an ideal sheaf for some closed scheme $S$.
	\end{enumerate}
\end{definition}
\noindent These three definitions are equivalents. 
\begin{remark} Let $\o\in H^0(M,\Om^1(L))$ be a section.
	\begin{enumerate}
		\item[(i)] The section $\o$ may be defined by a family of 1-forms 
		\[
		\o_{\a}\in \Om^1(U_{\a}),\quad \o_{\a}=\la_{\ab}\o_{\b}\mbox{ in }
		U_{\ab}={U_\a\cap U_\b},\quad L=[\la_{\ab}]\in \check{H}^{1}(\mathfrak{U},\O^{\ast}).
		\]
		\item[(ii)] The section $\o$ is a morphism of sheaves $\Theta\xrightarrow{\o} L$. 
		The kernel of $\o$ is \textit{the tangent sheaf} $\FF$. The image of $\o$ is a twisted ideal sheaf $\mathcal{N}=\mathcal{J}_{S_{\o}}\otimes L$. It is called the \textit{normal sheaf}. 
		
		\item[(iii)] As in the non singular case, the following equality of line bundles holds 
		\[ 
		K_M=\Om^n_M=det(\FF^{\ast})\otimes \mathcal{N}^{\ast}=K_{\FF}\otimes L^{-1},\quad det(N)\simeq L
		\]
		where $K_M, K_{\FF}=det(\FF^{\ast})$ are the canonical sheaf of $M$ and $\FF$ respectively.
	\end{enumerate}
\end{remark}
We denote by 
\[
\begin{array}{rcl}
\mathcal{F}(M,L) &=& \{[\o]\in\P H^0(M,\Om^1(L))\, | \,\cod(S_\o)\geq 2,\quad \o\wedge d\o=0 \,\} \\
\mathcal{F}(n,d) &=& \{[\o]\in\P H^0(\P^n,\Om^1(d+2))\,|\,\cod(S_\o)\geq 2,\quad \o\wedge d\o=0\}.
\end{array}
\] 
The number $d\geq0$ is called the \textit{degree of the foliation} represented by $\o$.

\subsection{Statement of the results}
\label{subsec:SR} 
In the sequel, $M$ is a compact complex manifold with $dim(M)\geq3$. We will use any of the above definitions for foliation. The singular set will be denoted by $S$.
Observe that $S$ decomposes as 
\[ 
S=\bigcup_{k=2}^{n} S_k\quad\text{where}\quad \cod(S_k)=k. 
\]
For a foliation $\mathcal{F}$ on $M$ represented by $\o\in\mathcal{F}(M,L),$ the Kupka set \cite{K,M} is defined by
\[
K(\o)=\{ p\in M\, |\: \o(p)=0,\, d\o(p)\neq 0\}.
\]
We recall that for points near $K(\o)$ the foliation $\mathcal{F}$ is biholomorphic to a product of a dimension one foliation in a transversal section by a regular foliation of codimension two \cite{K} and in particular we have $K(\o)\subset S_2$.  
\par In this note, we focus our attention on the set of non-Kupka points $N\!K(\o)$ of $\o$. The first remark is 
\[
N\!K(\o)=\{ p\in M | \o(p)=0,\, d\o(p)=0\}\supset S_3\cup\cdots\cup S_n.
\]
We analyze three cases, one in each section, the last two being the core of the work.

\begin{enumerate}
	\item $S_2=K(\o)$, then $N\!K(\o)=S_3\cup\dots\cup S_n$. 
	
	\item There is an irreducible component $Z\subset S_2$ such that $Z\cap K(\o)=\emptyset$.  
	
	\item For a foliation $\o\in\mathcal{F}(3,d)$. Let $Z\subset S_2$ be a connected component such that $Z\setminus Z\cap K(\o)$ is a finite set of points.
\end{enumerate}

The first case has been considered in \cite{Br1,C,C1,CS,CL}. Let $\o\in\mathcal{F}(n,d)$ be a foliation with  $K(\o)=S_2$ and connected, then $\o$ has a meromorphic first integral. In the generic case, the leaves  define a \textit{Lefschetz} or \textit{a Branched Lefschetz Pencil}. The non-Kupka points are isolated singularities $N\!K(\o)=S_n$. In this note, we present a new and short proof of this fact when the transversal type of $K(\o)$ is radial.

In the second section, we study the case of a non-Kupka irreducible component of $S_2$. These phenomenon arise naturally in the intersection of irreducible components of $\mathcal{F}(M,L)$.  The following result is a local normal form for $\o$ near the singular set and is a consequence of a result of F. Loray \cite{L}. 

\begin{maintheorem}\label{main_theorem}
	Let $\o\in \Om^{1}(\C^n,0)$, $n\geq3$, be a germ of integrable 1--form such that $\cod(S_\o)=2$, $0\in S_\o$ is a smooth point and $d\o= 0$ on $S_\o$. If $j^1_{0}\o\neq0$, then or either 
	\begin{enumerate}
		\item there exists a coordinate system $(x_1,\dots,x_n)\in\mathbb{C}^n$ such that $$j^{1}_0(\omega)=x_1dx_2+x_2dx_1$$ and $\mathcal{F}_{\o}$ is biholomorphic to the product of a dimension one foliation in a transversal section by a regular foliation of codimension two, or
		\item there exists a coordinate system $(x_1,\dots,x_n)\in\mathbb{C}^n$ such that $$\o= x_1dx_1+g_{1}(x_2)(1+x_1g_{2}(x_2))dx_2,$$ 
		such that $g_1, g_2\in\mathcal{O}_{\C,0}$ with $g_1(0)=g_2(0)=0$, or 
		\item $\omega$ has a non-constant holomorphic first integral in a neighborhood of $0\in\mathbb{C}^n$.
	\end{enumerate}
\end{maintheorem}  
The alternatives are not exclusives. The following example was suggest by the referee and show that the case (3) of Theorem \ref{main_theorem} cannot be avoid. 
\begin{example}
	Let $\omega$ be a germ of a 1-form at $0\in\mathbb{C}^3$ defined by $$\o=xdx+(1+xf)df$$ where $f(x,y,z)=y^2z$. We have
	$$\omega=xdx+2yz(1+xy^2z)dy+y^2(1+xy^2z)dz.$$ The singular set of 
	$\o$ is $\{x=y=0\}$ and $\{x=z=y^2=0\}$, therefore the singular set has an embedding point $\{x=z=y^2=0\}$ and $d\o$ vanish along $\{x=y=0\}$. We will show that $\omega$ has a holomorphic first integral $F$ in a neighborhood of $0\in\mathbb{C}^3$. In fact, let $t=f(x,y,z)=y^2z$ and set $\varphi:(\mathbb{C}^3,0)\to(\mathbb{C}^2,0)$ defined by $$\varphi(x,y,z)=(x,t).$$
	Let $\eta=xdx+(1+xt)dt$ be 1-form at $0\in\mathbb{C}^2$, note that $\omega=\varphi^{*}(\eta)$ and moreover $\eta(0,0)\neq 0$, this implies that $\eta$ is non-singular at $0\in\mathbb{C}^2$ and by Frobenius theorem $\eta$ has a holomorphic first integral $H(x,t)$ on $(\mathbb{C}^2,0)$. Defining $H_1(x,y,z):=H(x,f(x,y,z))=H(x,y^2z)$, we get $H_1$ is a holomorphic first integral for $\omega$ in a neighborhood of $0\in\mathbb{C}^3$.
\end{example}

We apply Theorem \ref{main_theorem} to a codimension one holomorphic foliation of the projective space with empty Kupka set. 

About the third case, consider a foliation $\o\in \mathcal{F}(3,d)$. Let $Z$ be a connected component of $S_2$. We count the number $|Z\cap NK(\o)|$ of non-Kupka points of $\o$ in $Z\subset S_2$.

\begin{secondtheorem}\label{Thm2}
	Let $\o\in \mathcal{F}(3,d)$ be a foliation and 
	$Z\subset S_2$ a connected component of $S_2$. Suppose that $Z$ is a local complete intersection and $Z\setminus Z\cap K(\omega)$ is a finite set of points, then
	$d\o|_{Z}$ is a global section of $K^{-1}_{Z}\otimes K_{\mathcal{F}}|_{Z}$ and the associated divisor $D_\o=\displaystyle\sum_{p\in Z}ord_p(d\o)\cdot p$ has degree 
	$$\deg(D_{\o}) = \deg(K_{\FF})-\deg(K_{Z}).$$ 
\end{secondtheorem}
\par Note that the section $d\o|_{Z}$ vanishes exactly in the non-Kupka points of $\omega$ in $Z$ then the above theorem determine the number $|Z\cap NK(\o)|$ (counted with multiplicity) of non-Kupka points of $\omega$ in $Z$.


\section{The singular set}
\label{sec:Sing}
Let $\omega\in\mathcal{F}(M,L)$ be a codimension one holomorphic foliation then singular set of $\omega$ may be written as 
\[S= \bigcup_{j=2}^n S_j\qquad \mbox{ where }\qquad \cod(S_j)=j.\] 
The fact that $K(\omega)\subset S_2$ implies that $S_3\cup\ldots\cup S_n\subset NK(\omega)$. To continue we focus in the components of singular set of $\omega$ of dimension at least three. 
\subsection{Singular set of codimension at least three}
We recall the following result due to Malgrange \cite{Ma}.

\begin{theorem}[Malgrange]
	Let $\o$ be a germ at $0\in \C^n$, $n\geq3$ of an integrable 1--form singular at $0$,  if $\cod(S_{\o})\geq 3$, then there exist $f\in\O_{\C^n,0}$ and $g\in \O^{\ast}_{\C^n,0}$ such that  
	\[
	\o = g df\quad\text{on a neighborhood of}\,\,\, 0\in\C^n.
	\] 
\end{theorem}


We have the following proposition.
\begin{proposition} Let $\o\in \mathcal{F}(M,L)$ be a foliation and let $p\in S_n$ an isolated singularity, then any germ of vector field tangent to the foliation vanishes at $p$.
\end{proposition}
\begin{proof}
	Let $\o=g df,\quad g\in \O^*_p,\, f\in \O_p$ be a 1--form representing the foliation at $p$. Let $\mathbf{X}\in\Theta_p$ be a vector field tangent to the foliation, i.e., $\o(\mathbf{X})=0$. 
	If $\mathbf{X}(p)\neq0$ there exists a coordinate system with $z(p)=0$ and $\mathbf{X}=\partial/\partial z_n$, then 
	\[
	0=\o(\mathbf{X})=g\cdot \left(\sum_{i=1}^n (\partial f/\partial z_i)dz_i(\partial/\partial z_n)\right)=g\cdot(\partial f/\partial z_n), \mbox{ therefore }\, \partial f/\partial z_n\equiv 0,
	\]
	and $f=f(z_1,\dots,z_{n-1})$, but this function does not have an isolated singularity.  
\end{proof}  

Now, we begin our study of the irreducible components of codimension two of the singular set of $\o$. Note that, given a section $\o\in H^0(M,\Om^1(L))$, along the singular set, the equation $\o_{\a}=\la_{\ab}\o_{\b}$ 
implies $d\o_{\a}|_{S}=(\la_{\ab}d\o_{\b})|_{S}.$ Then
\begin{equation}
\label{eq:do}
\{d\o_{\a}\}\in H^0(S,(\Om^2_M\otimes L)|_S).
\end{equation}

\subsection{The Kupka set}
These singularities has bee extensively studied and the main properties have been established in \cite{K,M}.

\begin{definition}
	For $\o\in \mathcal{F}(M,L)$. The Kupka set is
	\[ 
	K(\o)= \{ p\in M\,|\, \o(p)=0,\quad d\o(p)\neq0 \, \}.
	\]
\end{definition} 
The following properties of Kupka sets, are well known \cite{M}.
\begin{enumerate}
	\item $K(\o)$ is smooth of codimension-two.
	\item $K(\o)$ has \emph{local product structure\/} and the tangent sheaf $\FF$ is locally free near $K(\o)$.

	\item $K(\o)$ is subcanonically embedded and 
	\[
	\wedge^2 N_{K(\o)}=L|_{K(\o)},\quad K_{K(\o)}=(K_M\otimes L)|_{K(\o)}=K_{\FF}|_{K(\o)}. 
	\]
\end{enumerate}
Let $\o\in\mathcal{F}(n,d)$ be a foliation with $S_2=K(\o)$. 
By \cite{CS}, there exists a pair $(V,\si)$, where $V$ is a rank two holomorphic vector bundle and $\si\in H^0(\P^n,V)$, such that
\[
0\longrightarrow
{\mathcal{O}}\xrightarrow{\,\,\si \,\,} V
\longrightarrow {\mathcal{J}}_{K}(d+2)\rightarrow 0\quad \mbox{with}\quad \{\si=0\}=K  
\]
and the total Chern class
\[
c(V)=1 + (d+2)\cdot \mathbf{h} + \deg(K(\o))\mathbf{h}^2\in H^{\ast}(\P^n,\Z)\simeq \Z[\mathbf{h}]/\mathbf{h}^{n+1}.
\]
In 2009, Marco Brunella \cite{Br1} proved that following result, which in a certain sense say that the local transversal type of the singular set of foliation determines its behavior globally. Here we present a new proof of this fact. The techniques used in the proof could be of independent interest.

\begin{proposition}\label{KR}
	Let $\o\in\mathcal{F}(n,d)$ be a foliation with $S_2=K(\o)$, (connected if $n=3$) and of radial transversal type. Then $K(\o)$ is a complete intersection and $\o$ has a meromorphic first integral.  
\end{proposition}

To prove Proposition \ref{KR}, we requires the following lemma. This result may be well known but for lack of a suitable reference we include the proof in an appendix.

\begin{lemma}\label{Splitting} Let $F$ be a rank two holomorphic vector bundle over $\P^2$ with $c_1(F)=0$ and $c_2(F)=0$. Then $F\simeq \O\oplus \O$, is holomorphically trivial. 
\end{lemma}
Now, we prove Proposition \ref{KR}.

\begin{proof}[Proof of Proposition \ref{KR}] Let $(V,\si)$ be the vector bundle with a section defining the Kupka set as scheme. The radial transversal type implies \cite{CS}
	\[
	c(V)=1+(d+2)\cdot\mathbf{h}+\frac{(d+2)^2}{4}\cdot \mathbf{h}^2=\left(1+ \frac{(d+2)\cdot\mathbf{h}}{2}\right)^2\in H^{\ast}(\P^n,\Z)\simeq \Z[\mathbf{h}]/\mathbf{h}^{n+1}.
	\] 
	The vector bundle 
	$E=V(-\frac{d+2}{2})$, has $c_1(E)=0$ and $c_2(E)=0$. Let
	$\xi:\P^2\hookrightarrow \P^n$ be a linear embedding. By the preceding lemma we have  
	\[
	\xi^{\ast}E\simeq \O_{\P^2}\oplus \O_{\P^2}
	\]
	and by the Horrocks' criterion \cite[Theorem 2.3.2 pg. 22]{OSS}, 
	$E\simeq \O_{\P^n}\oplus \O_{\P^n}$ is trivial and hence $V$ splits as $\O_{\P^n}(\frac{d+2}{2})\oplus \O_{\P^n}(\frac{d+2}{2})$ and $K$ is a complete intersection. The existence of the meromorphic first integral follows from \cite[Th. A]{CL}.\end{proof}

If $\o$ is such that $K(\o)=S_2$ and connected, the set of non-Kupka points of $\o$ is  
\[
N\!K(\o)=S_3\cup\cdots\cup S_n.
\]

A generic rational map, that means, a \textit{Lefschetz or a Branched Lefschetz Pencil} 
$\va:\P^n\dasharrow \P^1$, has only isolated singularities away its base locus. The singular set of the foliation defined by the fibers of $\va$ is $S_n\cup S_2$. The Kupka set corresponds away from its base locus and $S_n=N\!K(\o)$ are the singularities as a map. $S_n$ is empty if and only if the degree of the foliation is $0$. The number 
$\ell(S_n)$ of isolated singularities counted with multiplicities can be calculated by\cite[Th. 3]{CSV}.
If $\o_p$ is a germ of form that defines the foliation at $p\in S_n$, we have 
\[
\ell(S_n)=\sum_{p\in S_n}\mu(\o_p,p),\quad \mu(\o,p)=dim_{\C}\frac{\O_p}{(\o_1,\dots,\o_n)},\quad \o_p=\sum_{i=1}^n \o_i dz_i.
\]
We have that $c_n(\FF)=\ell(S_n)$.

\section{Foliations with a non-Kupka component}
It is well known that $K(\o)\subset \{p\in M|\, j^1_p\o\neq0\}$, but the converse is not true. Our first result describes the singular points with this property.

\subsection{A normal form}
Now, we analyze the situation when there is an irreducible non-Kupka component of $S_2$.

\begin{proof}[Proof of Theorem \ref{main_theorem}]
	By hypotheses, $d\o(p)=0$ for any $p\in S_\o$.  Since 
	\[
	\o=\o_1+\cdots,\qquad d\o=d\o_1+\cdots =0,
	\]
	we get $d\o_1(p)=0$ for any $p\in S_\o$. Now, as $\o_1\neq0$ and $\cod(S_\o)=2$, we have $1\leq \cod(S_{\o_1})\leq 2$. We distinguish two cases.
	
	\begin{enumerate}
		\item $\cod(S_{\o_1})=2$: there is a coordinate system $(x_1,\ldots,x_n)\in\mathbb{C}^n$ such that $$\o_1=x_1dx_2+x_2dx_1.$$ 
		
		\item $\cod(S_{\o_1})=1$: there is a coordinate system $(x,\zeta)\in\C\times\C^{n-1}$ such that
		$x(p)=0$ and $\o_1=xdx$.
	\end{enumerate}
	
	The first case is known, the foliation $\mathcal{F}_{\o}$ is equivalent in a neighborhood of $0\in\C^n$ to a product of a dimension one foliation in a transversal section by a regular foliation of codimension two \cite[p. 31]{CM}. 
	
	In the second case, Loray's preparation theorem \cite{L}, shows that there exists a coordinate system $(x,\zeta)\in \C\times \C^{n-1}$, a germ $f\in \O_{\C^{n-1},0}$ with $f(0)=0$, and germs $g,h\in \O_{\C,0}$ such that the foliation is defined by the 1--form
	\begin{equation}\label{loray_form}
	\o=xdx+[g(f(\zeta))+xh(f(\zeta))]df(\zeta).
	\end{equation}
	Since $S_{\o_1}=\{x=0\}$ and $0\in S_{\o}$ is a smooth point, we can assume that ${S_{\o}}_{,p}=\{x=\zeta_1=0\}$, where $S_{{\o},p}$ is the germ of $S_{\o}$ at $p=0$. Therefore, 
	\[
	{S_{\o}}_{,p}=\{x=\zeta_1=0\}=\{x=g(f(\zeta))=0\}\cup
	\left\{x=\frac{\partial{f}}{{\partial{\zeta_1}}}=\cdots=\frac{\partial{f}}{{\partial{\zeta_{n-1}}}}=0 \right\}.
	\]
	Hence, either $g(0)=0$ and $\zeta_1|f$, or $g(0)\neq 0$ and $\zeta_1|\frac{\partial{f}}{\partial{\zeta_{j}}}$ for all $j=1,\ldots,n-1$.
	In any case, we have $\zeta_1|f$ and then $f(\zeta)=\zeta_1^k\psi(\zeta)$, where $\psi$ is a germ of holomorphic function in the variable $\zeta$; $k\in\mathbb{N}$ and $\zeta_1$ does not divide $\psi$. 
	We have two possibilities:
	
	\noindent
	{\textit{\bf $1^{st}$ case.--}}
	$\psi(0)\neq 0$. In this case, we consider the biholomorphism $$G(x,\zeta)=(x,\zeta_1\psi^{1/k}(\zeta),\zeta_2,\ldots,\zeta_n)=(x,y,\zeta_2,\ldots,\zeta_n)$$ where $\psi^{1/k}$ is a branch of the $k^{th}$ root of $\psi$, we get $f\circ G^{-1}(x,y,\zeta_2,\ldots,\zeta_n)=y^{k}$ and 
	\[
	G_{*}(\o)=xdx+(g(y^k)+xh(y^k))ky^{k-1}dy=xdx+(g_1(y)+xh_1(y))dy,
	\]
	where $g_1(y)=ky^{k-1}g(y^k)$, $h_1(y)=ky^{k-1}h(y^k)$.
	Therefore, $\tilde{\o}:=G_{*}(\o)$ is equivalent to $\o$ and moreover $\tilde{\o}$ is given by 
	\begin{equation}\label{equa_2}
	\tilde{\o}=xdx+(g_1(y)+xh_1(y))dy\quad\text{with}\quad S_{\tilde{\o}}=\{x=g_1(y)=0\}.
	\end{equation}
	Since $d\tilde{\o}=h_1(y)dx\wedge dy$ is zero identically  on ${\{x=g_{1}(y)=0\}}$, 
	we get $g_1|h_1$, so that $h_1(y)=(g_1(y))^{m}H(y)$, for some $m\in\mathbb{N}$ and such that $H(y)$ does not divided $g_1(y)$. Using the above expression for $h_1$ in (\ref{equa_2}), we have
	\[
	\tilde{\o}=xdx+g_{1}(y)(1+x(g_{1}(y))^{m-1}H(y))dy=xdx+g_{1}(y)(1+xg_{2}(y))dy,
	\]
	where 
	$g_2(y)=(g_{1}(y))^{m-1}H(y)$. Consider $\va:(\C,0)\times (\C^{n-1},0)\rightarrow (\C^{2},0)$ defined by $\va(x,\zeta)=(x,y)$, then 
	\begin{equation}\label{kupka_form}
	\o=\va^{\ast}(xdx+g_{1}(y)(1+xg_{2}(y))dy).
	\end{equation}
	
	\noindent
	{\textit{ \bf $2^{nd}$ case.--}}
	$\psi(0)=0$. We have $S_{{\o},p}=\{x=\zeta_1=0\}$ and 
	\begin{equation}\label{equa_1}
	\o=xdx+(g(\zeta^k_1\psi)+xh(\zeta^k_{1}\psi))d(\zeta^k_1\psi),  
	\end{equation}
	therefore 
	\begin{equation}
	\omega=xdx+(g(\zeta^k_1\psi)+xh(\zeta^k_{1}\psi))\zeta^{k-1}_1(k\psi d\zeta_1+\zeta_1d\psi).
	\end{equation}
	Note that $g(0)\neq 0$, otherwise $\{x=\zeta_1\psi(\zeta)=0\}$ would be contained in ${S_{\o}}_{,p}$, but it is contradiction because $S_{{\o},p}=\{x=\zeta_1=0\}\subsetneq\{x=\zeta_1\psi(\zeta)=0\}$. Furthermore $k\geq 2$, because otherwise $\zeta_1|\psi$. 
	Let $\va:(\C,0)\times (\C^{n-1},0)\rightarrow (\C^{2},0)$ be defined by $$\va(x,\zeta)=(x,\zeta^k_1\psi(\zeta))=(x,t),$$ then from (\ref{equa_1}), we get that 
	\[
	\o=\va^{\ast}(\eta), 
	\]
	where $\eta=xdx+(g(t)+xh(t))dt$. Since $\eta(0,0)=g(0)dt\neq0$, we deduce that $\eta$ has a non-constant holomorphic first integral $F\in\mathcal{O}_{\C^2,0}$ such that $dF(0,0)\neq 0$. Therefore, $F_1(x,\zeta)=F(x,\zeta^{k}_1\psi(\zeta))$ is a non-constant holomorphic first integral for $\omega$ in a neighborhood of $0\in\mathbb{C}^n$.  
\end{proof}

\subsection{Applications to foliations on $\P^n$}
In order to give some applications of Theorem \ref{main_theorem}, we need the Baum-Bott index associated to singularities of foliations of codimension one. 

Let $M$ be a complex manifold and let $\mathcal{G}_\o=(S,\mathcal{G})$ be a codimension one holomorphic foliation represented by $\o\in H^0(M,\Om^1(L))$. We have the exact sequence
\[
0\to \mathcal{G}\to \Theta_M \xrightarrow{\o} \mathcal{N}_{\mathcal{G}}\to 0,\quad 
\mathcal{N}_{\mathcal{G}}\simeq \mathcal{J}_S\otimes L.
\]
Set $M^0=M\setminus S$ and take $p_{0}\in M^0$. Then in a neighborhood $U_{\alpha}$ of $p_{0}$
the foliation $\mathcal{G}$ is induced by a holomorphic  $1$--form $\o_{\a}$ and there exists
a differentiable  $1$--form $\theta_{\a}$ such that
\[
d\o_\a=\theta_\a \wedge \o_\a
\]
Let $Z$ be an irreducible component of $S_2$. Take a generic point $p\in Z$, 
that is, $p$ is a point where $Z$ is smooth and disjoint from the other singular components.
Pick $B_{p}$ a ball centered at $p$ sufficiently small, so that
$S(B_p)$ is a sub-ball of $B_p$ of codimension $2$. Then the De Rham 
class can be integrated over an oriented $3$-sphere
$L_{p}\subset B^{*}_{p}$ positively linked with $S(B_{p})$:
\[
\BB(\mathcal{G}, Z)=\frac{1}{(2\pi i)^{2}}\int_{L_{p}}\theta\wedge d\theta.
\]
This complex number is the \textit{Baum-Bott residue of $\mathcal{G}$ along Z}. 
We have a particular case of the general Baum-Bott residues Theorem \cite{BB} reproved by Brunella and Perrone in \cite{BP}.

\begin{theorem}[Baum-Bott \cite{BB}]\label{BB} 
	Let $\mathcal{G}$ be a codimension one  holomorphic foliation on a  complex manifold $M$. Then 
	\[
	c_1(L)^2=c^{2}_{1}( \mathcal{N}_{\mathcal{G}})=\sum_{Z\subset S_2} \BB(\mathcal{G}, Z) [Z],
	\]
	where $\mathcal{N}_{\mathcal{G}}=\mathcal{J}_S\otimes L$ is the normal sheaf of $\mathcal{G}$ on $M$ and the sum is done over all irreducible components of  $S_2$.
\end{theorem}
In particular, if $\mathcal{G}$ is a codimension one foliation on $\mathbb{P}^{n}$ of degree $d$,  then the normal sheaf $\mathcal{N}_{\mathcal{G}}=\mathcal{J}_S(d+2)$ and the Baum-Bott Theorem looks as follows 
\[
\sum_{Z} \BB(\mathcal{G}, Z) \deg[Z]=(d+2)^{2}.
\]

Now, if there exist a coordinates system   $(U, (x, y,z_3,\dots,z_n))$ around $p\in Z \subset S_2$ such that $x(p) = y(p) = 0$ and $S(\mathcal{G})\cap U= Z \cap U = \{x=y=0\}$. Assume that  
\[
\o|_U=P(x,y)dy-Q(x, y)dx
\]
is a holomorphic 1-form representing $\mathcal{G}|_{U}$. Let $\theta$ be the $\mathcal{C}^{\infty}$ (1,0)-form on $U\setminus \ Z$ given by
\[
\theta=\frac{(\frac{\partial{P}}{\partial x}+\frac{\partial{Q}}{\partial y})}{|P|^{2}+|Q|^2}(\bar{P}dx+\bar{Q}dy).
\]
Since  
$d\omega=\theta\wedge\omega$, then  
\begin{equation}\label{baum_1}
\BB(\mathcal{G},Z)=\frac{1}{(2\pi i)^2}\int_{L_p}\theta\wedge d\theta=\mbox{Res}_{0}\displaystyle\left\{\frac{\tr\big(D\mathbf{X} \big)\,dx\wedge dy}{PQ}\right\},
\end{equation}
where $\mbox{Res}_{0}$ denotes the Grothendieck residue,  $D\mathbf{X}$ is the Jacobian of the holomorphic  map $\mathbf{X}=(P,Q)$.
It follows from of Grothendieck residues \cite[Chapter 5]{GH} that if   $D\mathbf{X}(p)$ is non-singular, then
\[
\BB(\mathcal{G},Z)=\frac{\tr(D\mathbf{X}(p))^{2}}{\det(D\mathbf{X}(p))}.
\]

In the situation explained above, the tangent sheaf $\mathcal{G}(U)$ is locally free and generated by the holomorphic vector fields
\[
\mathcal{G}(U)=\left\langle \mathbf{X}=P(x,y)\frac{\partial}{\partial x}+Q(x,y)\frac{\partial}{\partial y},\frac{\partial}{\partial z_3},\dots,\frac{\partial}{\partial z_n}\right\rangle
\]
and the vector field $\mathbf{X}$ carries the information of the Baum--Bott residues.

The next result, in an application of Theorem \ref{main_theorem}
\begin{theorem}
	Let $\o\in \mathcal{F}(M,L)$ be a foliation and $Z\subset S_2\setminus K(\o)$. Suppose that $Z$ is smooth and $j^1_p\o\neq0$ for all $p\in Z$, then $\BB(\mathcal{F}_\o,Z)=0$.
\end{theorem}

\begin{proof} We work in a small neighborhood $U$ of $p\in Z\subset M$. According to Theorem \ref{main_theorem} there exists a coordinate system $(x,y,z_3,\dots,z_n)$ at $p$ such that $Z\cap U=\{x=y=0\}$ and one has three cases. In the first case $\mathcal{F}_\o$ is the product of a dimension one foliation in a section transversal to $Z$ by a regular foliation of codimension two and $j^{1}_{p}(\o)=xdy+ydx$. In this case, it follows from (\ref{baum_1}) that $\BB(\mathcal{F}_{\o},Z)=0$. In the second case
	$$\o=xdx+g_1(y)(1+xg_2(y))dy,$$ where $g_1, g_2\in\mathcal{O}_{\C,0}$ and it follows from \cite[Lemma 3.9]{CL1} that  
	\[
	\BB(\mathcal{F}_\o,Z)=\res_{t=0}\left[\frac{(g_1(t) g_2(t))^2 dt}{g_1(t)}\right]=\res_{t=0}\left[ g_1(t) (g_2(t))^2\right].
	\]
	Since $g_1(y) (g_2(y))^2$ is holomorphic at $y=0$, we get $\BB(\mathcal{F}_\o,Z)=0.$ In the third case $\mathcal{F}_{\o}$ has a holomorhic first integral in neighborhood of $p$ and is known that $\BB(\mathcal{F}_\o,Z)=0.$
\end{proof}

The Baum-Bott formula implies the following result.

\begin{corollary}
	Let $\o\in \mathcal{F}(n,d)$, $n\geq 3$, be a foliation with $K(\o)=\emptyset$. Then there exists a smooth point $p\in S_2$ such that $j^1_p\o= 0$.
\end{corollary}
\begin{proof}
	If for all smooth point $p\in S_2$ one has $j^1_p\o\neq 0$, the above theorem shows that
	$\BB(\mathcal{F}_\o,Z)=0$ for all irreducible components $Z\subset S_2$. By Baum--Bott's theorem, we get
	\[
	0<(d+2)^2=\sum_{Z\subset S_2} \BB(\mathcal{F}_\o,Z)=0
	\] 
	which is a contradiction. Therefore there exists a smooth point $p\in S_2$ such that $j^1_p\o=0$. \end{proof}
In particular, if $\o\in \mathcal{F}(n,d)$, $n\geq 3$, is a foliation with $j^1_p\o\neq 0$ for any $p\in \P^n$, then its Kupka set is not empty.

\section{The number of non-Kupka points}

Through this section, we consider codimension one foliations on $\P^3$, but some results remain valid to codimension one foliations on others manifolds of dimension three.

\subsection{Simple singularities} 
Let $\o$ be a germ of 1--form at $0\in \C^3$. We define the \textit{rotational} of $\o$ as the unique vector field $\mathbf{X}$ such that
\[
rot(\o)=\mathbf{X} \Longleftrightarrow d\o=\im_{\mathbf{X}}dx\wedge dy\wedge dz,
\] 
moreover $\o$ is integrable if and only if $\o(rot(\o))=0$.

Let $\o$ be a germ of an integrable 1--form at $0\in \C^3$. We say that $0$ is a \textit{simple singularity} of $\o$ if $\o(0)=0$ and either $d\o(0)\neq0$ or $d\o$ has an isolated singularity at $0$. In the second case, these kind of singularities, are classified as follows
\begin{enumerate}
	\item \textit{Logarithmic}. The second jet $j^2_0(\o)\neq 0$ and the linear part of $\mathbf{X}=rot(\o)$ at $0$ has non zero eigenvalues.
	
	\item \textit{Degenerated}. The rotational has a zero eigenvalue, the other two are non zero and necessarily satisfies the relation $\la_1+\la_2=0$.
	
	\item \textit{Nilpotent}. The rotational vector field $\mathbf{X},$ is nilpotent as a derivation.
\end{enumerate}

The structure near simple singularity is known  \cite{CCGL}. If $p\in S$ is a simple singularity and $d\o(p)=0$, then $p$ is a singular point of $S$.

\begin{theorem}
	Let $\o\in \Omega^{1}(\C^3,0),$ $n\geq3$, be a germ of integrable 1-form such that $\o$ has a simple singularity at $0$ then the tangent sheaf $\FF=Ker(\o)$ is locally free at $0$ and it is generated by
	$\langle rot(\o),\mathbf{S} \rangle$, where $\mathbf{S}$ has non zero linear part.   
\end{theorem}
\begin{proof} Let $\o$ be a germ at $0\in \C^3$ of an integrable 1--form and $0$ a simple non-Kupka singularity. Then $0\in \C^3$ is an isolated singularity of $\mathbf{X}=rot(\o)$. 
	Consider the Koszul complex of the vector field $\mathbf{X}$ at $0$ 
	\[
	\mathbb{K}(\mathbf{X})_0: 0\to \Om^3_{\C^3,0}
	\xrightarrow{\imath_{\mathbf{X}}} \Omega^2_{\C^3,0}
	\xrightarrow{\imath_{\mathbf{X}}} \Omega^1_{\C^3,0}
	\xrightarrow{\imath_{\mathbf{X}}} \O_{\C^3,0} \to 0
	\]
	Since $\o(\mathbf{X})=0$, then $\o\in H^1(\mathbb{K}(\mathbf{X})_0)$ that vanishes because $\mathbf{X}$ has an isolated singularity at $0$.  
	Therefore, there exists $\theta\in \Omega^2_{\C^3,0}$ such that $\imath_{\mathbf{X}}\theta=\o$. 
	The map 
	$\Theta_{\C^3,0}\ni \mathbf{Z}\mapsto \imath_{\mathbf{Z}}dx\wedge dy\wedge dz\in \Omega^2_{\C^3,0}$
	is an isomorphism, hence
	\[
	\o=\imath_{\mathbf{X}}\theta,\quad\mbox{and}\quad \theta=\imath_{\mathbf{S}}dx\wedge dy\wedge dz,
	\quad\mbox{implies}\quad
	\o=\imath_{\mathbf{X}}\theta=\imath_{\mathbf{X}}\imath_{\mathbf{S}}dx\wedge dy\wedge dz
	\]
	and then, the vector fields $\{\mathbf{X},\mathbf{S}\}$ generate the sheaf $\FF$ in a neighborhood of $0$.
\end{proof}

%
%
%

Let $\o\in\mathcal{F}(3,d)$ be a foliation and $Z \subset S_2$ be a connected component of $S_2$. Assume that $Z$ is a local complete intersection and has only simple singularities. We will calculate the number $|N\!K(\o)\cap Z|$ of non-Kupka points in $Z$.


\begin{proof}[Proof of Theorem \ref{Thm2}]
	Let $\mathcal{J}$ be the ideal sheaf of $Z$. Since $Z$ is a local complete intersection, consider the exact sequence 
	\[
	0\rightarrow \mathcal{J}/\mathcal{J}^{2}
	\rightarrow \Om^1 \otimes \O_{Z}
	\rightarrow \Om^{1}_{Z}
	\rightarrow0
	\]
	Taking $\wedge^2$ and twisting by $L=K_{\P^3}^{-1}\otimes K_{Z}=K_{Z}(4)$  we get
	\[
	0\rightarrow 
	\wedge^{2}\mathcal{J}/\mathcal{J}^{2}\otimes L
	\rightarrow \Om^2_{\P^3}|_{Z}\otimes L 
	\rightarrow \cdots
	\]
	Since $Z\subset S$, the singular set, we have seen before that 
	\[
	d\o|_{Z}\in H^{0}(Z,\wedge^{2}(\mathcal{J}/\mathcal{J}^{2})\otimes L)
	\]
	Now, from the equalities of sheaves 
	\[
	K_{Z}^{-1}\otimes K_{\P^3}\simeq \wedge^{2}(\mathcal{J}/\mathcal{J}^{2}),\quad \mbox{and} \quad L\simeq K_{\P^3}^{-1}\otimes K_{\FF}
	\]
	we have 
	\[
	H^{0}(Z,\wedge^{2}(\mathcal{J}/\mathcal{J}^{2})\otimes L)=H^{0}(Z,K^{-1}_{Z}\otimes K_{\FF}|_{Z}),
	\]
	the non-Kupka points of $\o$ in $Z$ satisfies $d\o|_{Z}=0$, denoting $$D_\o=\sum_{p\in Z}ord_p(d\o)$$ the associated divisor to  $d\o|_{Z}$, one has
	\[
	\deg(D_{\o})=\deg(K_{\FF})-\deg(K_{Z}),
	\]
	as claimed.
\end{proof}

\begin{remark}
	The method of the proof works also in projective manifolds, and  does not depends on the integrability condition.
\end{remark}

\subsection{Examples}
We apply Theorem 2 for some codimension one holomorphic foliations on $\P^3$ and determine the number of non-Kupka points.

\begin{example}[Degree two Logarithmic foliations] 
	Recall that the canonical bundle of a degree two foliation of $\P^3$ is trivial. There are two irreducible components of logarithmic foliations in the space of foliations of $\P^3$ of degree two:
	$\mathcal{L}(1,1,2)$ and $\mathcal{L}(1,1,1,1)$. We analyze generic foliations on each component.
	
	$\mathcal{L}(1,1,2)$: let $\o$ be a generic element of $\mathcal{L}(1,1,2)$ and consider its singular scheme $S=S_2\cup S_3$. By \cite[Theorem 3]{CSV} $\ell(S_3)=2$. On the other hand, $S_2$ has three irreducible components, two quadratics and a line, the arithmetic genus is $p_a(S_2)=2$. Note that Theorem \ref{Thm2}, implies that the number $| N\!K(\o)\cap S_2|$, of non-Kupka points in $S_2$ is 
	\[
	| N\!K(\o)\cap S_2|=deg(D_\o)=\deg(K_{\FF})-\deg(K_{S_2})=-\chi(S_2)=2. 
	\]
	The non-Kupka points of the foliation $\mathcal{F}_{\o}$ are $|N\!K(\o)|=\ell(S_3)+|NK\cap S_2|=4$.
	
	$\mathcal{L}(1,1,1,1)$: let $\o$ be a generic element of $\mathcal{L}(1,1,1,1)$ then the tangent sheaf is $\O\oplus\O$ and the singular scheme $S=S_2$ \cite{GP} moreover consists of 6 lines given the edges of a tetrahedron, obtained by intersecting any two of the four invariant hyperplanes $H_i$. The arithmetic genus is $p_a(S_2)=3$, by Theorem \ref{Thm2}, $|N\!K(\o)|=|N\!K(\o)\cap S_2|=4$, corresponding to the vertices of the tetrahedron where there are simple singularities of logarithmic type. 
\end{example}

\begin{example} [The exceptional component $\mathcal{E}(3)$]
	The leaves of a generic foliation 
	$\o\in\mathcal{E}(3)\subset \mathcal{F}(3,2),$ are the orbits of an action of $\mathbf{Aff}(\C)\times\P^3\rightarrow\P^3$ and its tangent sheaf is $\O\oplus\O$ \cite{CCGL,GP}.
	Its singular locus $S=S_2$ has $deg(S)=6$ and three irreducible components: a line $L,$ a conic $C$ tangent to $L$ at a point $p$, and a twisted cubic $\Gamma$ with $L$ as an inflection line at $p$. The point $N\!K(\o)=L\cap C\cap\Gamma=\{p\}\subset S$ is the only non-Kupka point.
	
	The arithmetic genus is $p_a(S)=3$ and the canonical bundle of the foliation again is trivial, by Theorem \ref{Thm2}, the number of non-Kupka points $|N\!K(\o)|=4$. Therefore the non-Kupka divisor $N\!K(\o)\cap S=4p$. 
	If $\o$ represents the foliation at $p$, then $\mu(d\o,p)=\mu(rot(\o),p)=4.$  
\end{example}

\vspace{0.5cm}

\noindent\emph{ Acknowledgements.\/} 
The first author thanks the Federal University of Minas Gerais UFMG, IMPA, and IMECC--UNICAMP for the hospitality during the elaboration of this work. The third author thanks the IMCA-Per\'u for the hospitality. Finally, we would like to thank  the referee by the suggestions, comments and improvements to the exposition.

\section{Appendix}

We prove Lemma \ref{Splitting}.

\begin{proof}
	First, we see that $h^0(F)\geq1$. By Riemann--Roch--Hirzebruch, we have 
	\[
	\chi(F)=h^0(F)-h^1(F)+h^2(F)=[ch(F)\cdot Td(\P^2)]_2=2,
	\] 
	then 
	\[ h^0(F)+h^2(F)=[ch(F)\cdot Td(\P^2)]_2+h^1(F)\geq [ch(F)\cdot Td(\P^2)]_2= 2
	\]
	By Serre duality \cite{GH, OSS}, we get $h^2(F)=h^0(F(-3))$. Moreover  
	$h^0(F)\geq h^0(F(-k))$ for all $k>0$, hence $h^0(F)\geq1$. 
	Let $\tau\in H^0(F)$ be a non zero section, consider the exact sequence 
	\begin{equation}\label{eq:exs}
	0\longrightarrow
	{\mathcal{O}}\stackrel{\cdot\tau}{\longrightarrow} F
	\longrightarrow \mathcal{Q}\longrightarrow 0\quad \mbox{with} \quad 
	\mathcal{Q}=F/\mathcal{O}.
	\end{equation}
	The sheaf $\mathcal{Q}$ is torsion free, 
	therefore $\mathcal{Q}\simeq \mathcal{J}_{\Sigma}$ for some $\Sigma\subset \P^2$. The sequence (\ref{eq:exs}), is a free resolution of the sheaf $\mathcal{Q}$ with vector bundles with zero Chern classes. 
	From the definition of Chern classes for coherent sheaves \cite{BB}, we get $c(\mathcal{Q})=1$, in particular $\deg(\Sigma)=c_2(\mathcal{Q})=0$, we conclude that $\Sigma=\emptyset$ and  
	$\mathcal{Q}\simeq \O$. Then $F$ is an extension of holomorphic line bundles, hence it splits \cite[p. 15]{OSS}.\end{proof}

\bibliographystyle{amsalpha}

\end{document}